\documentclass[11pt,fleqn]{article}
\usepackage[english]{babel}
\usepackage{authblk}
\usepackage{a4wide}
\usepackage{xcolor}
\usepackage[utf8]{inputenc}
\usepackage[T1]{fontenc}
\usepackage{needspace}
\usepackage{amsmath}
\usepackage{amssymb}
\usepackage{amsthm}
\usepackage{mathtools}

\usepackage{stmaryrd}
\usepackage{eucal} 
\usepackage{comment}
\usepackage{hyperref}
\usepackage[shortlabels]{enumitem}
\setenumerate{itemsep=0.005ex}

\DeclareMathOperator{\ord}{ord}
\newcommand{\numtwo}{\nu_2}

\allowdisplaybreaks

\addtocounter{footnote}{1}

\theoremstyle{plain}
\newtheorem{theorem}{Theorem}
\newtheorem{lemma}{Lemma}

\newtheorem{corollary}{Corollary}

\newtheorem{algorithm}{Algorithm}
\theoremstyle{definition}
\newtheorem{remark}{Remark}
\newtheorem*{definition}{Definition}

\newcommand{\measure}[1]{\mu\!\left( #1 \right)}

\newcommand{\floor}[1]{\lfloor #1 \rfloor }

\renewcommand{\phi}{\varphi}

\newcommand{\C}{{\mathcal C}({\mathcal S})}

\renewcommand{\S}{{\mathcal S}}

\begin{document}
\title{Normal numbers in  sparse Cantor sets}
\author{Verónica Becher}
\author{Simón Lew Deveali}

\affil{
{\small {\tt vbecher@dc.uba.ar} \qquad
  {\tt sdeveali@dc.uba.ar}}
\\
Facultad de Ciencias Exactas y Naturales, \\
Universidad de Buenos Aires \& CONICET\\
Argentina
}
\maketitle

\begin{abstract}
We  consider Cantor-type sets of Hausdorff dimension zero,
 consisting of all numbers whose base-$2$ expansion can have a $1$ 
only at positions belonging to a given sparse set (local count at least 
 $\log k$ in every interval of length  $k)$. We prove that  the 
measure induced by independent, non-identically distributed Bernoulli digits
assigns full mass to numbers that are normal in every odd base.
The proof extends Schmidt's 1960 method to this Hausdorff zero-dimensional setting,
and we provide an explicit algorithmic construction of such numbers---yielding
the first known examples of numbers deterministic in base~$2$ yet normal in all odd bases.
This work supports our broader conjecture that given determinism in one base, normality in all multiplicatively independent bases is prevalent.
\end{abstract}
\medskip
\medskip

\tableofcontents

\section{Introduction and Statement of Results}

A base is an integer greater than or equal to $2$. 
For a real number $x$, its expansion in  base $b$ is  the sequence of 
digits $b_1(x), b_2(x), \ldots $ such that 
\[
x = \lfloor x\rfloor+ \sum_{j\geq 1} b_j(x) b^{-j}, \qquad b_j(x) \in \{0,\ldots b-1\}.
\]
where we assume the expansion does not end with a tail of $b-1$s.
A real number $x$ is simply normal in  base $b$ if each digit $\{0, \ldots, b-1\}$ occurs with 
asymptotic frequency $1/b$ in the expansion of $x$ in base $b$.
A real number $x$ is normal in base~$b$ if it is simply normal in base~$b^k$,  for every $k\geq 1$. 
For a presentation of normal numbers see~\cite{Bugeaud12,KuiNie74}.

We write  $\log$ for the natural logarithm. For integers $a,c$ we write $[a,c]$ for the set of integers between $a$ and $c$ inclusive.
For a set $\S$ of positive integers  let  $S(a,c)=\#({\mathcal S}\cap[a,c])$.

\begin{definition}[Sparse set]
\label{def:sparse}
A set  ${\mathcal S}$
of positive integers
  is  sparse 
  if it has density zero  and there exists a sparsity exponent $  \rho>0$, such that
\[
\liminf_{k\to\infty} \min_{\substack{0\le a\le k^{\rho}}} \frac{S(a,a+k)}{30\log k}>1.
\]
\end{definition}
The prime numbers, the  triangular numbers, and the squares are all examples of sparse sets.
Also 
$\{\lceil e^{j/100}\rceil: j\geq 1\}$ is sparse.
However, neither the powers of two, nor the factorials are sparse sets.

\begin{definition}[{Sparse Cantor set $\C$}]
For a sparse set ${\mathcal S}$ of positive integers,
the sparse Cantor set $\C$ is the set of real numbers in the unit interval  whose binary digits can be $1$ only at positions belonging to the sparse set ${\mathcal S}$. That is, 
\[\C=\Big\{x\in [0,1): x=\sum_{k\in \S} b_k 2^{-k}\Big\}, \quad (b_k\in\{0,1\})
\]
\end{definition}

We prove the following three results.

\begin{theorem}\label{thm:1}
Let ${\mathcal S}$ be a sparse set of positive integers.
There are uncountably many numbers in $\C$ that are normal in every  odd base. 
\end{theorem}

\begin{theorem}\label{thm:2}
Let ${\mathcal S}$ be a sparse set of positive integers.
For any even base $b$, there are uncountably many numbers in $\C$ that are not normal in any even base \(\le b\).
\end{theorem}

A function from non-negative integers to non-negative integers is computable
if it belongs to the smallest class of functions containing the constant functions,  successor, projections and it is closed under composition, primitive recursion, and unbounded minimization.
A set of positive integers is computable  if its characteristic function is a computable function defined for all positive integers. 
A real number is computable if 
there is a computable function that, for each $n$, gives the $n$-th digit of its expansion in some prescribed base.
The computable functions are realized by algorithms.

Since the proof of Theorem~\ref{thm:2} is constructive, 
given any computable set $\S$ with a computable sparsity exponent, 
the proof immediately yields  a distinct computable number for
each  computable increasing sequence of elements in $\S$
with exponential growth. In contrast, the proof of Theorem~\ref{thm:1} is nonconstructive; we therefore provide an explicit algorithm.

\begin{theorem}\label{thm:3}
Let ${\mathcal S}$ be a computable sparse set of positive integers with a computable sparsity exponent.
There is an algorithm  that
given a rational $\varepsilon$, $0<\varepsilon<2/3$, 
produces a  number $x(\varepsilon)$ in $\C$
expressed in base~two, that is normal in all odd bases.
\end{theorem}

The set $\C$ is generated by a  non-stationary iterated function system consisting of the two similarity maps
\[
f_0(x) = \frac{x}{2}, \qquad f_1(x) = \frac{x}{2} + \frac{1}{2}, \qquad x \in [0,1].
\]
However, unlike a classical self-similar set, the available maps depend on the digit position~$k$: at step $k$, the map $f_1$ is allowed if and only if $k \in \S$; if $k \notin \S$, then only the map $f_0$ can be used. 
Thus, $\C$ 
is not a self-similar set generated by a single fixed IFS applied uniformly at every scale,
 so our theorem does not follow from the sufficient condition for normality given by the geometrical criterion of Hochman and Shmerkin~\cite{HochmanShmerkin2015} for self-similar measures. 
It does not follow either from  the extension done by  Algom, Rodriguez Hertz, and Wang~\cite{ARW} that treat stationary self-conformal measures and  prove, under aperiodicity of the derivative cocycle, that such measures are pointwise  normal in all integer bases and Rajchman. Their result relies  on  the existence of a fixed IFS of conformal maps.
Measures supported on sets of Hausdorff dimension zero are not guaranteed to have a Fourier transform that decays sufficiently fast at infinity to ensure equidistribution.

We consider  Cantor-type sets $\C$ consisting of all numbers whose base-2 expansion can have a 1 only at positions belonging to a sparse set $\S$. Theorem~\ref{thm:1} establishes that the measure on $\C$ induced by independent, non-identically distributed Bernoulli digits---fair on the sparse set and zero outside it---is supported on numbers that are normal in all odd bases.
Theorem~\ref{thm:1} yields a novel case of  normality in a setting where the measure is not self-conformal and the underlying set $\C$ has Hausdorff dimension zero. 
An analogous situation occurs with Liouville numbers: although the set of Liouville numbers has Hausdorff dimension zero, almost all numbers in the support of Bluhm's measure, which consists entirely of Liouville numbers, are normal in every base~\cite{Bugeaud-liouville,liouville}.

The set $\C$ is made of just two digits, being $0$ everywhere except at sparse positions, where there is a choice between $0$ and $1$. This is the smallest amount of randomness for which we could prove that the Fourier transform decays fast enough to yield normality. 
Theorem~\ref{thm:1} asserts normality only to odd bases, keeping the proof simpler than if we had to account for multiplicative independent bases. Applying Schmidt's specialized tools from~\cite{Schmidt1962} it can be extended to all bases multiplicatively independent of $2$.
Besides, the proof of Theorem~\ref{thm:1} generalizes in a routine way to Cantor-type sets constructed from $b$ digits for $b \geq 3$, yielding non-normality to base $b$ and normality to all bases co-prime with $b$.

Theorem~\ref{thm:1} generalizes the singular case of Volkmann's theorem in~\cite{Volkmann1}, where he deals with a fixed  set $\S$ of triangular numbers and considers a Cantor-type set $\C$ with three digits. The numbers in $\C$ have either of two prescribed digits at the triangular positions, while at the non-triangular positions the third digit is fixed. For any base $b \geq 3$, Volkmann proves that uncountably many numbers in $\C$ are normal in all bases multiplicatively independent of~$b$. 
In contrast to Volkmann's work, Theorem~\ref{thm:1} applies not just when $\mathcal{S}$ is the set of triangular numbers, but to any sparse set whose gaps grow at most exponentially (the $n$th element of $\mathcal{S}$ is $O(e^{\rho n})$), with a uniform proof over all such $\mathcal{S}$.

Both, Volkmann's proof~\cite{Volkmann1} and our proof of Theorem~\ref{thm:1}, build on Schmidt's work in~\cite{Schmidt1962}, which itself generalizes Cassels's~\cite{Cassels1959} (see also~\cite{BecherSlaman2014,BBS}). Schmidt gives an explicit construction of uncountably many numbers that are normal in all bases in a set $X$ and not normal in any base in a set $Y$, where $X$ and $Y$ are closed under multiplicative dependence. For each base $b\geq 3$, his construction uses Cantor-type sets with two digits, of Hausdorff dimension $(\log 2)/(\log b)$. In contrast, Volkmann's proof and our proof of Theorem~\ref{thm:1} use Cantor-type sets of Hausdorff dimension zero. To obtain our result uniformly over all sparse sets $\S$, we redevelop Schmidt's~\cite[Hilfssatz~5]{Schmidt1962}, a key lemma that bounds the sum of a Riesz product to  yield the needed exponential sum estimate; Volkmann's proof does not require this.

\pagebreak
Theorem~\ref{thm:2}  proves a result  independent of that of Theorem~\ref{thm:1} concerning the failure of normality in even bases. The proof is constructive.
We do not know how to prove a similar result for lack of normality in $\C$ for odd bases. Topologically, fluctuating behaviour of the frequency of digits or blocks of digits is the most common form of non-normality in the full unit interval: in any given base, the set of real numbers exhibiting it is comeager in the full unit interval~\cite{Olsen,AveniLeonetti}. 
However, in the Cantor-type set $\C$ the usual topological arguments do not apply.

In Theorem~\ref{thm:3}
we give an algorithm that, given a sparse set $\S$ and a positive rational~$\varepsilon< 2/3$, produces a number $x(\varepsilon)$ in $\C$ that is normal in all odd bases. 
This answers a question of Pablo Shmerkin (personal communication, 2014), who asked for numbers whose digit and block frequencies converge to non-uniform limits in at least one base, yet are normal in the other multiplicatively independent bases.
We remark that Volkmann's work~\cite{Volkmann1} does not provide a computable construction in his setting. To our knowledge, Theorem~\ref{thm:3} is the first algorithm that yields  a number that has digit $0$ with asymptotic frequency $1$ in its base-$2$ expansion and is normal in all odd bases. Our algorithm 
can be generalized
to any choice of base in place of $2$ and any prescribed non-uniform asymptotic frequencies for digits or blocks of digits in that base, while preserving normality in all multiplicatively independent bases.

We base our algorithm in the  computable reformulation of Sierpi\'nski's construction~\cite{Sierpinski1917} of an absolutely normal number  done by the first author in~\cite{BF02}. 
We adapt it  in two aspects. First, we define the set of numbers in $\C$ whose initial segment deviates from normality in terms of averages of exponential sums, whereas  in~\cite{BF02} such sets are  unions of sets satisfying a discrete condition (digit occurrence counts). Second, the algorithm is adapted to make choices only at steps~$k$ where~$k$ is in the sparse set $\S$.
The resulting algorithm outputs the base-2  expansion of a real  number $x$ having $0$s at all positions not in $\mathcal{S}$. For the elements of $\mathcal{S}$ taken in increasing order, at the $i$-th element the algorithm demands  double  exponentially in~$i$ many steps to determine whether to output a $0$ or a $1$.
The main obstacle to proving Theorem~\ref{thm:3} with a direct adaptation of Schmidt's construction~\cite{Schmidt1962}---which could possibly yield a polynomial-time algorithm---lies in controlling the multiplicities of the residue classes $\floor{\ell r^n / 2^a} \bmod 2^k$ for fixed positive integers $a,k$, odd $\ell, r$, as $n$ ranges over $0, \dots, 2^k-1$.
\medskip

Regardless of the choice of the sparse set $\mathcal{S}$, all numbers in $\C$ have a base-$2$ expansion in which the digit $1$ occurs with asymptotic density $0$, hence the digit $0$ occurs with asymptotic density $1$. Such numbers, which are of course not normal in base $2$, form a subclass of the deterministic numbers in base $2$.
The notion of deterministic numbers was introduced by Weiss~\cite{Weiss1971} in the setting of dynamical systems: a number is deterministic in base $b$ if the orbit of its base-$b$ expansion under the shift map generates empirical measures (for each $n$, the average of point masses on the first $n$ shifted sequences $x, Tx, \dots, T^{n-1}x$) whose weak-$*$ limits all have zero Kolmogorov--Sinai entropy. Zero entropy means that the statistical description of the orbit never exhibits exponential growth in the number of distinguishable finite patterns.
Deterministic numbers admit an equivalent characterization due to Rauzy~\cite{Rauzy76}: a real number is deterministic in base $b$ if and only if adding it to any base-$b$ normal number $x$ preserves normality---that is, $x+y$ remains normal in base $b$. This builds on a preservation result by Spears and Maxfield~\cite{LSM}, who proved that adding a number whose base-$b$ expansion contains a single digit with asymptotic density $1$ preserves normality. Besides numbers with base-$2$ expansion with asymptotic density of $0$s equal to 1, examples of deterministic numbers in base $2$ include numbers with periodic base-$2$ expansions,   Sturmian numbers, and  numbers whose base-$2$ expansion is a paper-folding sequence.
Theorems~\ref{thm:1} and~\ref{thm:2} yield the following.

\begin{corollary}
There are uncountably many deterministic numbers in base $2$ that are normal in all odd bases; and for any even base $b$, uncountably many that are not normal in all even bases~$\leq b$ simultaneously. In both cases, we give an explicit construction of a countable family of  such numbers.
\end{corollary}


Bernay~\cite{Bernay} proved that the set of deterministic numbers has Hausdorff dimension~0.  We conjecture that,  among the deterministic numbers in base $2$, the normality in  all bases multiplicatively independent of~$2$ is a generic property.


\section{Proof of Theorem~\ref{thm:1}}

\begin{definition}[Measure $\mu$ on $\C$]
Given a sparse set ${\mathcal S}$ of positive integers,
we define a probability measure $\mu=\mu(\S)$
on the unit interval~$[0,1)$ equipped with the Borel $\sigma$-algebra, via the binary expansion of numbers.
The  numbers $x\in  \C $ are of the form $x=\sum_{k\in\S} b_k 2^{-k}$, where, 
for every  $k\ge1$, digit $b_k=0$ with probability  $1/2$ if $k \in \S$, and digit $b_k=0$ with probability $1$ otherwise. 
\end{definition}
So $\mu$ is the image of a product of independent but not identically distributed Bernoulli measures under the binary expansion map.
The measure $\mu$ is  singular and continuous  with respect to Lebesgue measure, and its support is  the set  $\C$.
The Fourier transform of  $\mu$ is
\[
\widehat{\mu}(h) = \prod_{k\in\S} e^{-\pi i h / 2^k} \cos(\pi {h}/{2^k}) 
= 
e^{ -\big(\pi i h \underset{\substack{\text{\tiny $k\!\in\!\S$}}}{\sum} 2^{-k}\big)} 
\prod_{k\in \S} \cos(\pi {h}/{2^k}).
\]
So,
\[
|\widehat{\mu}(h)| \leq  \prod_{k\in \S} \left|\cos(\pi {h}/{2^k})\right|.
\]
We show that for any sparse set $\mathcal{S}$, the measure $\mu$ is a Rajchman measure whose decay is rapid enough to satisfy the Davenport--Erdős--LeVeque theorem~\cite{LeVeque}   for all odd bases. It follows that $\mu$-almost every $x\in \C$ is normal in all odd bases. This is our strategy to prove Theorem~\ref{thm:1}.

\begin{lemma}
\label{lem:counting}
Let \(\mathcal S\) be a sparse set   and let $\rho$ be its sparsity exponent.  Then for all integers \(a\ge 0\) and all integers \(k\ge\delta_1(a)\),
\[
 S(a,a+k) > 30\log k,
\]
where  
\(
\delta_1(a) = \bigl(a^{1/\rho}+1\bigr)K,
\) with \(K\ge 2\) a constant depending only on \(\mathcal S\).
\end{lemma}
\begin{proof}
From the definition of a sparse set we can choose an integer \(K_0\) such that  
\[
S(a,a+k) > 30\log k \qquad\text{for all }k\ge K_0 \text{ and }0\le a\le k^\rho .
\]
Set \(K = \max\{K_0,\,2\}\), so that \(\log k>0\) for \(k\ge K\).  
Fix any \(a\ge 0\) and any \(k\ge (a^{1/\rho}+1)K\).  
Then \(k\ge K\ge K_0\) and \(k > a^{1/\rho}\), whence \(k^\rho > a\) and, because \(a\) is an integer, \(a\le k^\rho\).  
The required inequality follows.
\end{proof}

For an integer $x$ we write   $\numtwo(x)$ for the exponent of the highest power of $2$
dividing $x$.

\begin{lemma}\label{lem:residue}
Let integers $k\geq 1$,  $r\ge3 $ odd and $\ell$ odd.
 Among the $2^k$ numbers 
 \[
 \ell, \ell r, \ell r^2, \ldots ,\ell r^{2^k-1}
 \]
 at most $\delta_2$ lie in the 
same residue class modulo $2^{k}$,
where $\delta_2=2^{\nu_2(r^2-1)-1}$.
\end{lemma}

\begin{proof}
Let $h = \nu_2(r^2-1)$
For $k\ge 1$, let $d_k = \ord_{2^k}(r)$ be the multiplicative order of
 $r$ modulo~$2^k$.
Since $(\mathbb{Z}/2^k\mathbb{Z})^\times$ is a $2$-group, 
$d_k$ is a power of $2$, say $d_k=2^m$, with $m<k$.
We bound $m$ from below. Since $r^{2^m}\equiv 1\pmod{2^k}$, 
we have $2^k\mid r^{2^m}-1$.
Using the Lifting-The-Exponent lemma we show by induction that 
$\nu_2(r^{2^t}-1)=h+t-1$ for $t\ge1$. The case $t=1$ is the definition of $h$.
For $t\ge2$, write
$r^{2^t}-1 = (r^{2^{t-1}}-1)(r^{2^{t-1}}+1).
$
Since $r^{2^{t-1}}\equiv1\pmod4$, $\nu_2(r^{2^{t-1}}+1)=1$ 
and by the inductive hypothesis 
$\nu_2(r^{2^{t-1}}-1)=h+t-2$. So 
\[
\nu_2(r^{2^t}-1)=  \nu_2(r^{2^{t-1}}-1)+ \nu_2(r^{2^{t-1}}+1)= h+t-1.
\]
With $t=m$, we get $h+m-1\ge k$, hence $m\ge k-h+1$ and $d_k\ge2^{k-h+1}$.
For a fixed residue $A$ modulo $2^k$, the number of $n\in\{0,\dots,2^k-1\}$ with 
$\ell r^n\equiv A\pmod{2^k}$ is either $0$ or exactly $2^k/d_k$, because $n\mapsto r^n$ is a homomorphism 
$\mathbb{Z}\to(\mathbb{Z}/2^k\mathbb{Z})^\times$ with kernel~$d_k\mathbb{Z}$.
Thus, the maximum number of $n$ yielding the same residue is at most $2^k/d_k\le2^{h-1}$.
Hence, $\delta_2=2^{h-1}=2^{\nu_2(r^2-1)-1}$.
\end{proof}
By digits in base $2$ we mean the numbers $0,1$. 
For an integer $q\ge1$, each integer $x$ such that 
$0\le x<2^q$ has a $q$-bit binary representation
$x=\sum_{i=1}^{q} c_i 2^{i-1}$ with $c_i\in\{0,1\}$.
The digit pairs of $x$ (for this fixed $q$) are the pairs $(c_{q},c_{q-1}),\dots,(c_2,c_1)$.
For a non-negative integer $a$, an \emph{$a$-obedient digit
  pair} is a pair of digits $c_{i-a}c_{i-1-a}$ 
  that are not both equal to $0$ nor both   equal to $1$, and for which the index  $i$ is sparse. 
We let $z_a(x)$ denote the number of $a$-obedient digit pairs $c_{i-a}c_{i-1-a}$ of $x$.

\begin{lemma}\label{lem:obedient}
Let $\mathcal S$ be a sparse set. For all integers $a\ge 0$ and all integers $q\ge \delta_3(a)$,
\[
\#\Bigl\{0\le x< 2^q : z_a(x) < \Bigl\lfloor 2.9\log q\Bigr\rfloor\Bigr\}
< 2^q q^{-2}
\]
where \(\delta_3(a)=\max(\delta_1(a), 10^{30})\)
with $\delta_1(a)$ determined in  \autoref{lem:counting}.
\end{lemma}

\begin{proof}
Lemma~\ref{lem:counting} gives $S(a,a+q) > 30\log q$ for $q\ge\delta_1(a)$.
Choose a maximum subset ${\mathcal I}\subseteq (\S\cap[a+2,a+q])$ with no two consecutive elements.
Let $I=\#{\mathcal I}$. Then,
\[
 I \ge \frac12\bigl(S(a,a+q)-2\bigr) > 15\log q - 1.
\]
Since any two distinct elements of $\mathcal I$ differ by at least $2$, the first coordinates  of the $a$-obedient digit pairs $(i-a, i-1-a)$, with $i \in \mathcal I$, are necessarily pairwise distinct.
Set $L = \bigl\lfloor \frac{6}{31}I \bigr\rfloor$. Since $L\le I/2$, the binomial tail bound yields
\[
\#\{0\leq x<2^q: z_a(x)<L\}\le 2^{q-I}\sum_{t=0}^{L-1}\binom{I}{t}
\le L\cdot 2^{q-I}\binom{I}{L}.
\]
Using the entropy inequality $\binom{n}{k}\le 2^{nH(k/n)}$ and $H(6/31)<0.709$ (so $1-H(6/31)>0.291$),
\[
L\cdot 2^{q-0.291I}\le q\cdot 2^{q-0.291I}.
\]
Since $I > 15\log q-1$, $0.291I > 4.365\log q-0.291$.
For $q\ge 10^5$,
\[
4.365\log q-0.291 > 3\log_2 q = \frac{3}{\ln2}\log q \approx 4.328\log q,
\]
so $0.291I > 3\log_2 q$ and $2^{-0.291I} < q^{-3}$.  Consequently
\[
\#\{0\leq x<2^q: z_a(x)<L\}
< q\cdot 2^q\cdot q^{-3} = 2^q q^{-2}.
\]
Finally, for all $q\ge 10^{30}$,
\[
\frac{6}{31}I > \frac{90}{31}\log q - \frac{6}{31}
            \ge 2.9\log q.
\]
Hence $L = \bigl\lfloor \frac{6}{31}I \bigr\rfloor \ge \lfloor 2.9\log q\rfloor$.
\end{proof}

\begin{lemma}\label{lem:geometric}
Let $\mathcal S$ be a sparse set, let $a\geq 0$ be an integer and let $r\ge 3$ and  $\ell>0$ be odd integers.  Then, for every $N\ge \delta_4(a)$,
\[
\#\Bigl\{ 0\leq n < N : z_a(\ell r^{\,n}) < \bigl\lfloor 2.9\log\lfloor\log_2 N\rfloor \bigr\rfloor \Bigr\}
\le (\log 2)^2\,2^{\nu_2(r^2-1)+2}\frac{N}{(\log N)^2},
\]
where \(\delta_4(a)=2^{\delta_3(a)}\)
with $\delta_3(a)$ determined in Lemma~\ref{lem:obedient}.
\end{lemma}

\begin{proof}
Set $q = \lfloor\log_2 N\rfloor$ and $L = \bigl\lfloor 2.9\log q \bigr\rfloor$.
For all $N\ge \delta_4(a)$ we have $q \ge \delta_3(a)$.
Applying Lemma~\ref{lem:obedient} gives
\[
\#\{\,0<x\le 2^q : z_a(x) < L\,\} < 2^q q^{-2}.
\]
For each $n$,  $0\leq n < N$, let $Y_n = \ell r^{\,n} \bmod 2^q$.  Then $z_a(\ell r^{\,n}) \ge z_a(Y_n)$ and $Y_n \neq 0$ because $\ell,r$ are odd.
The sequence $(Y_n)_{n\ge 0}$ is periodic, and its period divides $2^q$.  
Therefore, as $n$ ranges from $0$ to $N-1$, the sequence covers at most two full periods of length $2^q$; that is,  at most $2 \cdot 2^q$ terms.
By Lemma~\ref{lem:residue} each residue class modulo $2^q$ occurs at most $\delta_2=2^{\nu_2(r^2-1)-1}$ times
so for every $x$,
\[
\#\{0\leq n < N : Y_n = x\} \le 2\delta_2 .
\]
Consequently,
\begin{align*}
\#\{0\leq n < N : z_a(\ell r^{\,n}) < L\}
&\le \sum_{\substack{x=0\\ z_a(x)<L}}^{2^q-1} 2\delta_2
   \le 2\delta_2 \cdot 2^q q^{-2} \
\le\  2\delta_2\, N\, q^{-2}.
\end{align*}
Now $q \ge \frac{\log N}{2\log 2}$ for every $N\ge 4$, so $q^{-2} \le (2\log 2)^2 (\log N)^{-2}$.
Thus
\[
\#\{n < N : z_a(\ell r^{\,n}) < L\} \le 2(2\log 2)^2\delta_2 \frac{N}{(\log N)^2}.
\]
Since 
  $L = \bigl\lfloor 2.9\log q \bigr\rfloor$ and $q = \lfloor\log_2 N\rfloor$  the lemma is proved.
\end{proof}

\begin{lemma}\label{lem:combined}
Let $\mathcal S$ be a sparse set, let $a\geq 0$ be an integer and let $r\ge 3$ and  $\ell>0$ be odd integers.
For every $N\ge \delta_4(a)$,
\[
\sum_{n=0}^{N-1}\;\prod_{\substack{k>a\\ k\in\mathcal S}}
\bigl|\cos(\pi \ell r^{\,n}/2^{\,k-a})\bigr|
\le 2^{\nu_2(r^{2}-1)+2}\,\frac{N}{(\log N)^{1.005}} .
\]
with $\delta_4(a)$ determined in Lemma~\ref{lem:geometric} .
\end{lemma}

\begin{proof}
Set $q = \lfloor\log_2 N\rfloor$ and $L = \bigl\lfloor 2.9\log q \bigr\rfloor$.
Split $\{0,\dots,N-1\}$ into
\[
\mathcal{A} = \{ 0\leq n < N : z_a(\ell r^{\,n}) < L \},\qquad
\mathcal{B} = \{ 0\leq n < N : z_a(\ell r^{\,n}) \ge L \}.
\]
\noindent\textit{Contribution of $\mathcal{A}$.}
Each factor in the product is $\le 1$.  By Lemma~\ref{lem:geometric} for all $N\ge \delta_4(a)$,
\[
\#\mathcal{A} \le (\log 2)^2\,2^{\nu_2(r^2-1)+2} \frac{N}{(\log N)^2},
\]
Hence,
\[
\sum_{n\in\mathcal{A}} \prod_{\substack{k>a\\ k\in\mathcal S}} |\cos(\pi \ell r^{\,n}/2^{\,k-a})|
\le K_\mathcal A \frac{N}{(\log N)^2}, \qquad\text{with } K_\mathcal A = (\log 2)^2\,2^{\nu_2(r^2-1)+2}.
\]
\noindent\textit{Contribution of $\mathcal{B}$.}
For $n\in\mathcal{B}$ write $x = \ell r^{\,n} = \sum_{j\ge1} c_j 2^{\,j-1}$ ($c_j\in\{0,1\}$).
For any sparse index $i>a$ with $i\le a+q$, the two highest bits of $x\bmod 2^{i-a}$ are $c_{i-a}$ and $c_{i-a-1}$.
If the pair is obedient, the fractional part of $x/2^{i-a}$ lies in $[\frac14,\frac34)$, where $|\cos(\pi x/2^{i-a})|\le 1/\sqrt2$.
By definition, $z_a(x)$ counts exactly such indices $i$ with $a+2\le i\le a+q$.
The indices $i>a+q$ and $a<i<a+2$ contribute at most a factor $1$.
Since $n\in\mathcal{B}$, we have $z_a(x)\ge L$, and therefore
\[
\prod_{\substack{k>a\\ k\in\mathcal S}} |\cos(\pi x/2^{k-a})|
\le \Bigl(\frac1{\sqrt2}\Bigr)^{\!L}
= 2^{-L/2}.
\]
Now $L = \bigl\lfloor 2.9\log q \bigr\rfloor \ge 2.9\log q - 1$, so
\[
2^{-L/2} \le 2^{-(2.9\log q - 1)/2}
        = \sqrt2\; q^{-(2.9/2)\ln 2}.
\]
Since $q = \lfloor\log_2 N\rfloor \ge \frac{\log N}{2\log 2}$ for $N\ge 4$, we obtain
\[
\sqrt{2}q^{-\frac{2.9}{2}\ln 2 } \le \sqrt{2}(2\log 2)^{\frac{2.9}{2}\ln 2 }\, (\log N)^{-\frac{2.9}{2}\ln 2 } \le K_\mathcal B\, (\log N)^{-1.005 },
\]
with $K_\mathcal B = \sqrt{2}(2\log 2)^{\frac{2.9}{2}\ln 2 }$.
Thus, for every $n\in\mathcal{B}$,
\[
\prod_{\substack{k>a\\ k\in\mathcal S}} 
|\cos(\pi \ell r^{n}/2^{k-a})|
\le K_\mathcal{B}\, (\log N)^{-1.005}.
\]
Consequently,
\[
\sum_{n\in\mathcal{B}} \prod_{\substack{k>a\\ k\in\mathcal S}}  |\cos(\pi \ell r^{n}/2^{k-a})|
\le K_\mathcal{B}\,\frac{N}{(\log N)^{1.005}} .
\]
\noindent\textit{Combining the contributions of ${\mathcal A}$ and ${\mathcal B}$.}
For $N\ge N_1 = \max(N_0, 4)$,
\[
\sum_{n=0}^{N-1} \prod_{\substack{k>a\\ k\in\mathcal S}}  |\cos(\pi \ell r^{n}/2^{k-a})|
\le K_\mathcal{A}\,\frac{N}{(\log N)^2} + K_\mathcal{B}\,\frac{N}{(\log N)^{1.005}} .
\]
For $N\ge 3$ we have $(\log N)^{-2} \le (\log N)^{-1.005}$.
Hence, the whole sum is bounded by $(K_\mathcal{A} + K_\mathcal{B}) N/(\log N)^{1.005}\le 2^{\nu_2(r^{2}-1)+2}  N/(\log N)^{1.005}$.
\end{proof}

\begin{lemma}
\label{lem:valcount}
Let $r\ge 3$ be odd. For every integer $d\ge 1$ and integer $N\ge 1$,
\[
\#\{1\le g\le N : \nu_2(r^{g}-1)=d\}
\;\le\; 2^{\,\nu_2(r^{2}-1)-1}\,\frac{N}{2^{\,d}} .
\]
\end{lemma}

\begin{proof}
The $2$-adic valuation of $r^{g}-1$ is given by Lifting‑the‑Exponent lemma,
\[
\nu_2(r^{g}-1)=
\begin{cases}
\nu_2(r-1) & \text{if $g$ is odd},\\[4pt]
\nu_2(r^{2}-1)+\nu_2\!\bigl({g}/{2}\bigr) & \text{if $g$ is even}.
\end{cases}
\]
Let $G=\{1\le g\le N : \nu_2(r^{g}-1)=d\}$.
If $d<\nu_2(r^{2}-1)$ and $d\neq\nu_2(r-1)$, there is no $g$ in $G$, hence $\#G=0$ and the lemma is proved.
$\nu_2(r^{g}-1)=d$,
Now assume the set $G$  is non‑empty. Then either $d=\nu_2(r-1)$ or
$d\ge\nu_2(r^{2}-1)$.
If $d\ge\nu_2(r^{2}-1)$, then $g$ must be even and
$\nu_2(g/2)=d-\nu_2(r^{2}-1)$; hence $g$ is a multiple of
$2^{\,d-\nu_2(r^{2}-1)+1}$.  The number of such $g\le N$ is at most
\[
\frac{N}{2^{\,d-\nu_2(r^{2}-1)+1}}
= 2^{\,\nu_2(r^{2}-1)-1}\,\frac{N}{2^{\,d}} .
\]
If $d=\nu_2(r-1)$, then $g$ must be odd, so there are at most $N$ such $g$.  In this case the right‑hand side of the claimed inequality equals
$N\cdot 2^{\,\nu_2(r+1)-1}$, which is greater than or equal to $N$.
\end{proof}

\begin{lemma} 
\label{lem:L2-key}
Let 
\(\mathcal S\) be a sparse set.
Then, for every odd integer \(r\ge 3\)
, every \(N\ge \delta_7\)
and every non‑zero integer \(h\) with \(|h|\le\log\log N\),
\[
\int_{\C}\Bigl|\frac1N\sum_{j=1}^{N} e^{2\pi i r^{j} h x}\Bigr|^{2}\,\mathrm d\mu(x)
\le \frac{2^{\,2\nu_2(r^{2}-1) + 4}}{(\log N)^{1.005}},
\]
where 
\(\delta_7\) is the smallest integer greater than or equa to  \(3\) satisfying 
\(\delta_7 \ge \delta_4\!\bigl(\lfloor 3\log_2\log \delta_7\rfloor + 1\bigr)\), with~$\delta_4$ from \autoref{lem:geometric}.
\end{lemma}

\begin{proof}
Write \(h = 2^{\nu_2(h)}u\) with \(u\) odd.  Expanding the square and using
\(|\widehat\mu(t)| = \prod_{k\in\mathcal S}|\cos(\pi t/2^{k})|\) gives
\[
\int_{\C}\Bigl|\sum_{j=1}^{N} e^{2\pi i r^{j} h x}\Bigr|^{2}\,\mathrm d\mu(x)
\le \sum_{p=1}^{N} \sum_{q=1}^{N}\;
      \prod_{k\in\mathcal S}\bigl|\cos\bigl(\pi h(r^{p}-r^{q})/2^{k}\bigr)\bigr|.
\]
The diagonal \(p=q\) contributes \(N\), since \((r^{p}-r^{q})=0\) and \(\cos(0)=1\).

For \(p>q\) set \(g = p-q\ge 1\) and \(d(g)=\nu_2(r^{g}-1)\).
Factoring \(h(r^{p}-r^{q}) = 2^{\nu_2(h)+d(g)} \cdot \ell_g r^{q}\) with \(\ell_g\) odd,
the product becomes 
\[
\prod_{\substack{k>\nu_2(h)+d(g)\\ k\in\mathcal S}}
\bigl|\cos\bigl(\pi \ell_g r^{\,q}/2^{\,k-\nu_2(h)-d(g)}\bigr)\bigr|.
\]
Thus, the off‑diagonal terms become
\[
2\sum_{g=1}^{N-1}\;\sum_{q=1}^{N-g}\;
  \prod_{\substack{k>\nu_2(h)+d(g)\\ k\in\mathcal S}}
  \bigl|\cos(\pi\ell_g r^{\,q}/2^{\,k-\nu_2(h)-d(g)})\bigr|.
\]
Set the threshold
$B = \bigl\lfloor 2\log_2(\log N)\bigr\rfloor.$

{\em Small valuations \(d(g)\le B\):}
For
$g$ such that \(d(g)\le B\)
we extend the inner sum harmlessly to the full block
\(q=0,\dots,N-1\) and apply Lemma~\ref{lem:combined} with \(a = \nu_2(h)+d \) and odd \(\ell = \ell_g\),
\begin{align*}
\sum_{q=0}^{N-g}\;\prod_{\substack{k>\nu_2(h)+d(g)\\ k\in\mathcal S}}
      \bigl|\cos(\pi\ell_g r^{\,q}/2^{\,k-\nu_2(h)-d(g)})\bigr|
&\le 
\sum_{q=0}^{N-1}\;\prod_{\substack{k>\nu_2(h)+d(g)\\ k\in\mathcal S}}
      \bigl|\cos(\pi\ell_g r^{\,q}/2^{\,k-\nu_2(h)-d(g)})\bigr|
      \\
&\le 2^{\nu_2(r^{2}-1)+2}\,\frac{N}{(\log N)^{1.005}},
\end{align*}
provided \(N \ge \delta_4(\nu_2(h)+d(g))\).
Let \(\delta_7\) be the smallest integer \(\ge 3\) satisfying 
\linebreak
\(\delta_7 \ge \delta_4\!\bigl(\lfloor 3\log_2\log \delta_7\rfloor + 1\bigr)\).  
Since the function \(f(N) = \delta_4\!\bigl(\lfloor 3\log_2\log N\rfloor + 1\bigr)\) grows slower than any positive power of \(N\) (indeed \(\log f(N) = O((\log\log N)^{1/\rho})\)), the inequality \(N \ge f(N)\) persists for all \(N \ge \delta_7\).  
Hence for every \(N \ge \delta_7\) we have \(N \ge \delta_4(\nu_2(h)+d(g))\), and \autoref{lem:combined} applies.
Now we count how many \(g\) have a given valuation $d(g)=d$.  By Lemma~\ref{lem:valcount},
\[
\#\{g\le N : d(g)=d\} \le 2^{\nu_2(r^{2}-1)-1}\,\frac{N}{2^{\,d}} .
\]
Summing over \(d=1,\dots,B\) gives at most \(2^{\nu_2(r^{2}-1)-1}N\) such \(g\).
Consequently, the total contribution of all small‑valuation terms is at most
\[
 2^{\nu_2(r^{2}-1)+2}\frac{N}{(\log N)^{1.005}}\cdot 2^{\nu_2(r^{2}-1)-1}N
   = 2^{2\nu_2(r^{2}-1)+1}\,\frac{N^{2}}{(\log N)^{1.005}} .
\]

\begin{samepage}
{\em Large valuations \(d(g) > B\).}
Here we bound each cosine factor by \(1\); each inner sum is \(\le N\).
The number of $g$ such that \(d(g)>B\) is
\[
\#\{g\le N : d(g)>d\} 
\leq \sum_{d>B} 2^{\nu_2(r^{2}-1)-1}\frac{N}{2^{\,d}}
= 2^{\nu_2(r^{2}-1)-1}\frac{N}{2^{\,B}} .
\]
Since \(B = \lfloor 2\log_2(\log N)\rfloor\), we have
\(2^{\,B} \ge \frac12 (\log N)^{2}\) for all \(N\ge 3\).
Thus, 
the total contribution of all large valuation terms is at most 
\[
2^{\nu_2(r^{2}-1)-1}\frac{N}{2^{\,B}}\,N
\le 2^{\nu_2(r^{2}-1)}\frac{N^{2}}{(\log N)^{2}}
\le 2^{\nu_2(r^{2}-1)}\frac{N^{2}}{(\log N)^{1.005}} ,
\]
where the last inequality uses \((\log N)^{2} \ge (\log N)^{1.005}\) for \(N\ge 3\).
\end{samepage}

\noindent{\em Combining the estimates.}
The double sum (before the symmetry factor \(2\)) satisfies
\[
\sum_{g=1}^{N-1}\sum_{q=1}^{N-g} \cdots
\le \bigl(2^{2\nu_2(r^{2}-1)+1}+ 2^{\nu_2(r^{2}-1)}\bigr)\frac{N^{2}}{(\log N)^{1.005}} .
\]
Multiplying by \(2\) gives the total off‑diagonal bound
\[
2\sum_{g=1}^{N-1}\sum_{q=1}^{N-g} \cdots
\le 2^{2\nu_2(r^{2}-1)+3}\frac{N^{2}}{(\log N)^{1.005}} .
\]
Finally, adding the diagonal \(N\) and dividing by \(N^{2}\) yields, for
\(N\ge \delta_7 \ge 3\),
\[
\int_{\C}\Bigl|\frac1N\sum_{j=1}^{N} e^{2\pi i r^{j} h x}\Bigr|^{2}\,\mathrm d\mu(x)
\le \frac{1}{N} + \frac{2^{2\nu_2(r^{2}-1)+3}}{(\log N)^{1.005}} 
\le \frac{2^{2\nu_2(r^{2}-1)+4}}{(\log N)^{1.005}} .
\]
because for \(N\ge 3\) we also have \({1}/{N} \le 1/(\log N)^{1.005}\).
\end{proof}
\bigskip

\begin{proof}[Proof of Theorem~\ref{thm:1}]
Fix a sparse set $\S$ and the measure $\mu$
Fix a non‑zero integer $h=2^c u$ with $u$ odd.
By Weyl's criterion and the Davenport--Erdős--LeVeque lemma~\cite{LeVeque},
it suffices to show

\[
\sum_{N\geq 1}\frac{1}{N}\int_{\C}
\Bigl|\frac{1}{N}\sum_{n=1}^{N}e^{2\pi i r^n h x}\Bigr|^{2}\mathrm{d}\mu(x)<\infty .
\]
By Lemma~\ref{lem:L2-key} for every $N\ge\delta_7$ with $|h|\le\log\log N$,

\[
\int_{\C}\Bigl|\frac1N\sum_{j=1}^{N} e^{2\pi i r^{j} h x}\Bigr|^{2}\,\mathrm d\mu(x)
\le \frac{2^{\,2\nu_2(r^{2}-1) + 4}}{(\log N)^{1.005}} .
\]
Choose $N_0 = N_0(h)$ satisfying $N_0\ge\delta_7$,
the constant set in Lemma~\ref{lem:L2-key},
and $|h|\le\log\log N_0$; then the estimate holds for all $N\ge N_0$.  
Hence,

\[
\sum_{N\geq 1}\frac{1}{N}\int_{\C}\Bigl|\frac1N\sum_{j=1}^{N} e^{2\pi i r^{j} h x}\Bigr|^{2}\,\mathrm d\mu(x)
\le 
\sum_{N=1}^{N_0-1}\frac{1}{N} \;+\; 2^{\,2\nu_2(r^{2}-1) + 4}\!\!\sum_{N\geq N_0} \frac{1}{N(\log N)^{1.005}} < \infty ,
\]
because the second series converges by the integral test.  
Thus, the whole series converges, and by the Davenport--Erdős--LeVeque lemma,

\[
\lim_{N\to\infty}\frac{1}{N}\sum_{n=1}^{N} e^{2\pi i r^n h x} = 0,
\qquad\text{for $\mu$-almost every } x.
\]
Since this holds for every non‑zero integer $h$, a countable intersection of full $\mu$-measure sets shows that for $\mu$-almost every $x\in\C$ the sequence $(r^nx)_{n\ge1}$ is equidistributed modulo~$1$.
\end{proof}
\bigskip

\pagebreak

\section{Proof of \autoref{thm:2}}

The following lemma derives immediately from an example due to Bugeaud~\cite[p. 935]{Bugeaud-expansions}.

\begin{lemma}
\label{lem:ceros}
Let $\S$ be a sparse set, let \(b\) be an even integer 
and let   \(\alpha\) be a real number such that  \(0 < \alpha < \log_b 2\). Let \((n_i)_{i \ge 1}\) be a strictly increasing sequence of positive integers in $\S$ such that  
$n_i > ( n_{i-1})/\alpha$ for all sufficiently large  $i$.
Let $x = \sum_{i\geq 1} 2^{-n_i}$.  
Then, for every  $i$ large enough, among the first \(\lfloor \alpha n_i \rfloor\) digits of the base-\(b\) expansion of \(x\), there are at most \(n_{i-1}\) non‑zero digits.
\end{lemma}
\begin{proof}
Set $c = \log_b 2$ and write $b = 2m$, 
 where $m = b/2$ is an integer. Note that $0 < \alpha < c < 1$ and $\log_b m = \log_b(b/2) = 1 - c$.
For any positive integer $n$, we have
\[
2^{-n} = \frac{m^n}{b^n}.
\]
The integer $m^n$ has at most
\(
L(n) = \lfloor n\log_b m \rfloor + 1 = \lfloor n(1-c) \rfloor + 1
\)
digits in base $b$, \[
b_{L(n)}\ldots b_1.
\]
Dividing  $m^n$ by $b^n$ shifts $m^n$  base-$b$ expansion   $n$ places to the right, 
\[
0. \overbrace{0 \ldots 0 }^{n-L(n)} b_{L(n)}\ldots \ldots b_1.
\]
So the base-$b$ expansion of $2^{-n}$  can have non-zero digits only at positions
\[
 n - L(n) + 1,\; n - L(n) + 2,\; \dots,\; n.
\]
The first  of these positions is
\[
n - \lfloor n(1-c) \rfloor = \lceil n c \rceil \ge  n c.
\]
Let $i_0$ be such that $n_i > n_{i-1}/\alpha$ for all $i \ge i_0$. We  prove the statement for all sufficiently large $i$; in particular  assume $i > i_0$. Then $n_{i-1} < \alpha n_i$, and since $n_{i-1}$ is an integer,
\[
n_{i-1} \le \lfloor \alpha n_i \rfloor.
\]
Now consider the base-$b$ expansion of $x = \sum_{j\geq 1} 2^{-n_j}$. We examine the contributions to the first $\lfloor \alpha n_i \rfloor$ digits.

\medskip
\noindent\textit{Terms with $j < i$:} For each such term, the non-zero digits of $2^{-n_j}$ lie at positions $\le n_j \le n_{i-1}$. Since $n_{i-1} \le \lfloor \alpha n_i \rfloor$, all these digits fall within the first $\lfloor \alpha n_i \rfloor$ positions. The total number of non-zero digits contributed by these terms is at most $\sum_{j=1}^{i-1} L(n_j)$.

\medskip
\noindent\textit{Terms with $j \ge i$:} The first non-zero digit of $2^{-n_j}$ is at position $\ge c n_j \ge c n_i$. Since $c > \alpha$, we have $c n_i > \alpha n_i$, so the first digit of any term with $j \ge i$ occurs strictly after position $\lfloor \alpha n_i \rfloor$. Hence these terms contribute no non-zero digits to the first $\lfloor \alpha n_i \rfloor$ positions.

\medskip
It remains to bound $\sum_{j=1}^{i-1} L(n_j)$. Using $L(n_j) \le n_j(1-c) + 1$, we obtain
\[
\sum_{j=1}^{i-1} L(n_j) \le (1-c)\sum_{j=1}^{i-1} n_j + (i-1).
\]
We now estimate $\sum_{j=1}^{i-1} n_j$. From the growth condition $n_k > n_{k-1}/\alpha$, which is equivalent to $n_{k-1} < \alpha n_k$, we have for all $k$ with $i_0 < k \le i-1$,
\[
n_{k-1} < \alpha n_k.
\]
Iterating this yields, for $i_0 \le j < i$,
\[
n_j < \alpha^{i-1-j} n_{i-1}.
\]
Split the sum into two parts:
\[
\sum_{j=1}^{i-1} n_j = \sum_{j=1}^{i_0-1} n_j + \sum_{j=i_0}^{i-1} n_j.
\]
The first sum is a constant $C$ (independent of $i$). For the second sum,
\[
\sum_{j=i_0}^{i-1} n_j < n_{i-1} \sum_{j=i_0}^{i-1} \alpha^{i-1-j} < n_{i-1} \sum_{k\geq 0} \alpha^k = \frac{n_{i-1}}{1-\alpha}.
\]
Thus
\[
\sum_{j=1}^{i-1} n_j < C + \frac{n_{i-1}}{1-\alpha}.
\]
Substituting back,
\[
\sum_{j=1}^{i-1} L(n_j) < (1-c)\left(C + \frac{n_{i-1}}{1-\alpha}\right) + (i-1).
\]
Set $\delta = \frac{c-\alpha}{1-\alpha} > 0$. Then
\[
(1-c)\frac{n_{i-1}}{1-\alpha} = \left(1 - \frac{c-\alpha}{1-\alpha}\right)n_{i-1} = n_{i-1} - \delta n_{i-1}.
\]
Hence
\[
\sum_{j=1}^{i-1} L(n_j) < n_{i-1} - \delta n_{i-1} + (1-c)C + (i-1).
\]
Since $\alpha < 1$, the condition $n_i > n_{i-1}/\alpha$ forces the sequence $(n_i)$ to grow at least geometrically, so $n_{i-1} \to \infty$ exponentially fast as $i \to \infty$. Consequently, the term $-\delta n_{i-1}$ dominates $(1-c)C + (i-1)$, and for all sufficiently large $i$ we have
\[
\delta n_{i-1} \ge (1-c)C + (i-1),
\]
which implies
\[
\sum_{j=1}^{i-1} L(n_j) < n_{i-1}.
\]
Since the left-hand side is an integer, we conclude
\[
\sum_{j=1}^{i-1} L(n_j) \le n_{i-1}.
\]
Therefore, for all sufficiently large $i$, the first $\lfloor \alpha n_i \rfloor$ digits of $x$ in base $b$ contain at most $n_{i-1}$ non-zero digits.
\end{proof}

\begin{proof}[Proof of \autoref{thm:2}]

By \autoref{lem:ceros} if we choose the sequence $(n_i)$	so that $\frac{n_i-1}{n_i} \to 0$ (possible because we may take it to grow arbitrarily fast), then the bound on non‑zero digits forces the frequency of zeros in the first\(\lfloor \alpha n_i \rfloor\) digits to tend to $1$. Then $x=\sum_{i\geq 1}2^{-n_i}$ is not normal in base $b$ nor in any other even base $b'< b$ because \(\log_b 2<\log_{b'} 2\).

 Finally, since \(\mathcal{S}\) is sparse, it contains uncountably many  strictly increasing sequences \((m_k)\) with \(m_{k+1} > (m_k)/\alpha\) and \(m_{k-1}/m_k\to 0\). For any infinite such subsequence \((m_{k_i})\) define \(x = \sum_{i\geq 1} 2^{-m_{k_i}}\in\mathcal{C}(\mathcal{S})\). 
 Since different subsequences give different numbers, we conclude that  \(\mathcal{C}(\mathcal{S})\) contains uncountably many numbers that are not normal in every even base \(\le b\).
\end{proof}

\section{Proof of Theorem~\ref{thm:3}}

The measure $\mu$ on the unit interval $[0,1]$ supported on
\[
{\mathcal C}({\mathcal S})= \left\{ \sum_{n \in S} b_n 2^{-n} : b_n \in \{0,1\} \right\}.
\]
For a dyadic interval $I=(a2^{-n},(a+1)2^{-n}) $
 such that the binary expansion of $a2^{-n} = 0.a_1\cdots a_n$, 
\[
\mu\bigl(I\bigr) =
\begin{cases}
2^{-S(1,n)}, & \text{if for all $i$ such that $1\leq i\leq n$ }, a_i = 1 \text{ implies } i \in S \\[4pt]
0, & \text{otherwise}.
\end{cases}
\]
Let $I=(0.a_1...a_n, 0.a_1...a_n +2^{-n})$
 be a positive-$\mu$-measure dyadic interval  in ${\mathcal C}({\mathcal S})$
and let 
$I^0=(0.a_1...a_n0, 0.a_1...a_n 1)$ and $I^1=(0.a_1...a_n1, 0.a_1...a_n1+2^{-(n+1)})$, 
the left and right halves of $I$ respectively.
Then,
\begin{align*}
\mu(I^0) &=
\begin{cases}
\mu(I), & n+1 \notin \S,\\[4pt]
\mu(I)/2, & n+1 \in \S,
\end{cases}
\\[4pt]
\mu(I^1) &=
\begin{cases}
0, & n+1 \notin \S,\\[4pt]
\mu(I)/2, & n+1 \in \S.
\end{cases}
\end{align*}
Therefore,
both halves have positive measure exactly when $ n+1 \in \S$.
If $n+1 \notin \S$, the half ending in $1$ has measure zero.

Any open set in the  Cantor-type set $\C$ 
can be written as a countable disjoint union of sets 
of the form $\C \cap I$, where $I$ is a dyadic interval.
Since $\mu$ is countably additive,
if $X\subseteq \C$ and  $\mu(X) > 0$ then at least one of those dyadic intervals must have positive $\mu$-measure.  Endpoints have $\mu$-measure zero because they are intersections of nested 
 dyadic intervals with measures tending to $0$; hence the result holds for all interval  types.

For  every integer $m\ge 1$ we  let $N_n$ denote the number of terms to be considered in  the exponential sum and set
\[
N_m = \bigl\lfloor 2^{m^{\alpha}} \bigr\rfloor,
\]
where
 $\alpha = 0.999$ (actually any fixed $\alpha$ such that ${1}/{1.005}<\alpha<1$ works). 
So, $N_m$ is subexponential in $m$,
$\log N_m $ = 
$ m^{\alpha}\log 2 + O(1)$, and $\log\log N_m = \log(m^{\alpha}\log 2 + O(1)) \le 2\log m$ for sufficiently large $m$.

\begin{lemma} 
\label{lem:L2-block}
For every odd \(r\ge 3\), for every \(m\ge \delta_9(r)\) and every non‑zero
integer \(h\) with \(|h|\le\log\log N_{m}\),
\[
\int_{\C}\Bigl|\frac1{N_{m}}\sum_{j=1}^{N_{m}} e^{2\pi i r^{\,j} h x}\Bigr|^{2}\,\mathrm d\mu(x)
\le \frac{C_{r}}{m^{1.003}} .
\]
where \(\delta_9(r)= \max\{\big( \log_2 \delta_7(r) \big)^{\!{1}/{\alpha}},3\}\) and \(C_{r} = 2^{\,2\nu_2(r^{2}-1) + 4} / (\frac12 \log 2)^{1.005}\).
\end{lemma}

\begin{proof}
Let \(\delta_9(r)= \max\{\big( \log_2 \delta_7(r) \big)^{\!{1}/{\alpha}},3\}\).
Recall \(N_{m} = \lfloor 2^{m^{\alpha}}\rfloor\) with \(\alpha = 0.999\).
 Since \(m\ge \delta_9(r)\) then
\(N_{m} \ge \delta_7(r)\). By Lemma~\ref{lem:L2-key} with \(N = N_{m}\),
\[
\int_{\C}\Bigl|\frac1{N_{m}}\sum_{j=1}^{N_{m}} e^{2\pi i r^{\,j} h x}\Bigr|^{2}\,\mathrm d\mu(x)
\le \frac{2^{\,2\nu_2(r^{2}-1) + 4}}{(\log N_{m})^{1.005}} .
\]
Since \(N_{m} \ge 2^{m^{\alpha}}-1\) and for \(m\ge 3\)
we have \(2^{m^{\alpha}} \ge 2^{3^{0.999}} > 8\), it follows that for
all \(m\ge 3\),
\[
\log N_{m} \ge m^{\alpha}\log 2 - \log 2 \ge \frac12 m^{\alpha}\log 2.
\]
Therefore, for \(m\ge \delta_9(r)\ge3\),
\[
\frac{2^{\,2\nu_2(r^{2}-1) + 4}}{(\log N_{m})^{1.005}}
\le \frac{2^{\,2\nu_2(r^{2}-1) + 4}}{(\frac12 \log 2)^{1.005}}\,
   \frac{1}{m^{1.005\alpha}}
\le \frac{C_{r}}{m^{1.003}} ,
\]
where \(C_{r} = 2^{\,2\nu_2(r^{2}-1) + 4} / (\frac12 \log 2)^{1.005}\) and we used
\(1.005\alpha = 1.003995 > 1.003\) to obtain for all \(m\ge 1\),  $1/m^{1.005\alpha} \le 1/m^{1.003}$.
\end{proof}

Define  for each non-zero-integer $h$, odd $r\ge 3$ and integer $m\ge 1$,
\begin{align*}
f_{h,r,m}(x)&=\frac1{N_m}\sum_{j=1}^{N_m} e^{2\pi i r^{\,j} h x},
\\
\Lambda_{r,m}&= \bigcup_{|h|=1}^{ \log\log N_m}
\Bigl\{ x\in \C:
\Bigl|f_{h,r,m}(x)\Bigr|
> \frac1{\log\log N_m} \Bigr\}.
\end{align*}

\begin{remark}
\label{rem:1}
Each $\Lambda_{r,m}$ is a finite union of  intervals. 
For each $h,r,m$, the boundary condition \(|f_{h,r,m}(x)| = 1/\log\log N_m\) is equivalent to a polynomial equation in \(\cos(2\pi x)\) with rational coefficients; hence the endpoints are algebraic numbers and computable.
so the endpoints are computable; consequently, $\mu(\Lambda_{r,m})$ is   computable uniformly in $r,m$.
\end{remark}

\begin{lemma}\label{lem:measure-Lambda}
For all \(m\ge \delta_9(r)\),
\[
\mu(\Lambda_{r,m}) \le C'_r\frac{(\log m)^{3}}{m^{1.003}} .
\]
where \(C'_r=16C_r\) with $C_r$ and $\delta_9(r)$ determined in \autoref{lem:L2-block}.
\end{lemma}

\begin{proof}
By the generalized Markov/Chebyshev inequality,
\[
 \mu (\{x\in X\,:\,\, |f(x)| > T\})\ \ < \frac{1}{T^2} \int _{X}|f(x)|^{2}\,d\mu(x). 
\]
Let $T=1/(\log \log N_m)$.
By Lemma~\ref{lem:L2-block}, for all \(m\ge \delta_9(r)\ge 3\) and
every non‑zero integer \(h\) with \(|h|\le\log\log N_m\),
\[
\int_{\C}|f_{h,r,m}(x)|^{2}\,\mathrm d\mu(x) \le C_r/m^{1.003}.
\]
Since the set \(\Lambda_{r,m}\) is the union over at most \(2\log\log N_m\)
integers \(h\) with \(1\le|h|\le\log\log N_m\),
\begin{align*}
\mu(\Lambda_{r,m})
&\le \frac{1}{(\log\log N_m)^{2}}
\sum_{|h|=1}^{\log\log N_m}
\int_{\C}|f_{h,r,m}(x)|^{2}\,\mathrm d\mu(x)
\\
&\le (\log\log N_m)^{2} \cdot 
2\log\log N_m \cdot \frac{C_r}{m^{1.003}}
= 2C_r\,\frac{(\log\log N_m)^{3}}{m^{1.003}} 
\\
&\leq 2C_r\cdot 8\,\frac{(\log m)^{3}}{m^{1.003}}
= 16\,C_r\,\frac{(\log m)^{3}}{m^{1.003}} .
\end{align*}
The last inequality follows 
for \(m\ge 3\) because
\(N_m \le 2^{m^{\alpha}}\) with \(\alpha<1\), so
\(\log N_m 
\le m\log 2\),
hence \(\log\log N_m 
\le  2\log m\).
%
\end{proof}

\begin{remark}\label{rem:2}
The series $\sum_{m\geq 1} (\log m)^{3}/m^{1.003}$ converges (its terms are $O(1/m^{1.003})$).  Its tail can be bounded effectively, for instance by comparing with the integral $\int (\log x)^{3}/x^{1.003}\,dx$, which is a computable decreasing function of $m$ that can be approximated from above by rationals to any desired precision.
\end{remark}

Let $\varepsilon$ a rational , $0<\varepsilon<2/3$, which is an input of  Algorithm~\ref{alg:1} and remains fixed.
For each base odd base $r$ we define $m_0(r)$ as the least integer such that 
\[
\sum_{m\geq m_0(r)} C'_r\frac{(\log m)^{3}}{m^{1.003}} < \frac{\varepsilon}{2^{r}}.
\]
where $C_r'$ is determined in Lemma~\ref{lem:measure-Lambda}.
We define

\begin{align*}
\Delta &= \bigcup_{\substack{r\ge3\\ r\text{ odd}}}\;
        \bigcup_{m\geq m_0(r)} \Lambda_{r,m},
\\s &= \sum_{\substack{r\ge3\\ r\text{ odd}}}\;
      \sum_{m\geq m_0(r)} \mu(\Lambda_{r,m}).
\end{align*}
Then,
\[
\mu(\Delta) \le s < \sum_{\substack{r\geq 3\\ r\text{ odd}}}\frac{\varepsilon}{2^{r}} = \frac{\varepsilon}{6} < \varepsilon .
\]
For an integer $t\ge 3$, define the approximating set $\Delta_t$, the partial sum $s_t$ and the tail $u_t$,
\begin{align*}
\Delta_t &= \bigcup_{\substack{r=3\\ r\text{ odd}}}^{\lfloor \log t\rfloor}\;
          \bigcup_{m=m_0(r)}^{t} \Lambda_{r,m},
\\
s_t &= \sum_{\substack{r=3\\ r\text{ odd}}}^{\lfloor \log t\rfloor}\;
      \sum_{m=m_0(r)}^{t} \mu(\Lambda_{r,m}),
\\
u_t&  = s-s_t.
\end{align*}
Clearly, for every $t\geq 3$, $\mu(\Delta_t)\leq s_t$ and 
 $u_t>0$ with   $\lim_{t\to\infty}u_t = 0$.

\begin{remark}\label{rem:3}
Let's see that $u_t$ is computable.
Since the computable numbers form a field, any finite sum of computable numbers is computable.
Each $\mu(\Lambda_{r,m})$ is computable, so for each $t$, $s_t$ is computable.
The sequence $(s_t)_{t\ge 3}$ is nondecreasing and bounded above by $s$,
and $s = \lim_{t\to\infty} s_t = \sup_t s_t$.
A standard theorem in computable analysis states that the supremum of
a computable, increasing, bounded sequence of computable reals is
computable exactly when there
exists a computable function $\ell \mapsto t_\ell$ such that
$s - s_{t_\ell} < 2^{-\ell}$ for all~$\ell$.
For each $t \ge 3$,
\[
u_t = s - s_t
    = \Big(\sum_{\substack{r > {\lfloor \log t\rfloor} \\ r\ \text{odd}}} \;
      \sum_{m\geq m_0(r)} \mu(\Lambda_{r,m})\Big)
    + \Big(\sum_{\substack{r=3 \\ r\ \text{odd}}}^{\lfloor \log t\rfloor} \;
      \sum_{m > \max(m_0(r),t)}\mu(\Lambda_{r,m})\Big).
\]
The first double sum is  bounded by the tail of the convergent geometric
series. The exponential decay in 
$r$ converts the logarithmic threshold 
$r>\log t$ into $1/t$ decay,
\[
\sum_{\substack{r > {\lfloor \log t\rfloor} \\ r\ \text{odd}}} \frac{\varepsilon}{2^r}
= \frac{2\varepsilon}{3} \cdot 2^{-{\lfloor \log t\rfloor}}<(4\varepsilon/3t^{\log 2}).
\]
The second double sum decays  to zero polynomially with a very small exponent, 
\begin{align*}
\sum_{\substack{r=3 \\ r\ \text{odd}}}^{\lfloor \log t\rfloor} \sum_{m>\max\{m_0(r),t\}} \mu(\Lambda_{r,m})
&\leq 
\sum_{\substack{r=3 \\ r\ \text{odd}}}^{\lfloor \log t\rfloor} C'_r\sum_{m>\max\{m_0(r),t\}} \frac{(\log m)^3}{m^{1.003}} 
\\&
\leq 
K \frac{(\log  t)^3}{t^{0.003}} 
\sum_{\substack{r=3 \\ r\ \text{odd}}}^{\lfloor \log t\rfloor} 
C'_r \\
&\leq K'\frac{(\log t)^{8}}{t^{0.003}}
\end{align*}
where 
$C'_{r} =  16 2^{\,2\nu_2(r^{2}-1) + 4} / (\frac12 \log 2)^{1.005}$ and
  $K' $ is a positive constant.
Hence we can effectively
choose $t_\ell$ such that $u_{t_\ell} < 2^{-\ell}$, providing the required
computable modulus of convergence.  
Therefore $s$ is computable.
Finally, for each fixed $t$, we have
$u_t = s - s_t$.
Since
the computable  numbers are a field and both $s$ and $s_t$ are computable reals,  their difference
$u_t$ is computable.
\end{remark}

\begin{lemma}\label{lem:props}
For any interval $I$ and any integers $t,\ell$ with $3\leq t\le \ell$,
\begin{enumerate}
\item $\mu(\Delta\setminus\Delta_t) \le u_t$.\label{item:1}
\item $\mu(\Delta_\ell\setminus\Delta_t) \le u_t - u_\ell$.\label{item:2}
\item $\mu(\Delta\cap I) \le \mu(\Delta_t\cap I) + u_t$.\label{item:3}
\item $\mu(\Delta_\ell\cap I) \le \mu(\Delta_t\cap I) + u_t - u_\ell$.\label{item:4}
\end{enumerate}
\end{lemma}
\begin{proof}
These follow directly from the definitions of $\Delta,\Delta_t$, $s_t$, $u_t$ and  subadditivity.
\end{proof}
\medskip

We now give the algorithm that proves  Theorem \ref{thm:3}.

\needspace{16\baselineskip}
\begin{algorithm}  \label{alg:1} \em 
Let $k_1 < k_2 < k_3< \cdots$ be the elements of $\S$ in increasing order.
Given positive rational $\varepsilon<2/3$, the following  algorithm defines the binary expansion $x =x(\varepsilon)= 0.b_1 b_2 b_3\ldots$ digit by digit.

\begin{itemize}[ itemindent=-1cm]
\item[]{\em Initialisation $i=0$.}
For $j=1,\ldots, k_1-1$ set $b_j=0$.
Set $I_0 = [0,2^{-(k_1-1)})$, so   $|I_0|=2^{-(k_1-1)}$ and $\mu(I_0)=1$.
Choose $p_0$ large enough so that $u_{p_0} < \varepsilon$. 

\item[]{\em Step $i\ge 1$.}
For $j=k_{i-1}+1,..,k_i-1$ 
set $b_j=0$.
The current  interval is  $I_{i-1}$, with $|I_{i-1}|=2^{-(k_{i}-1)}$ and $\mu(I_{i-1}) = 2^{-(i-1)}$.

Split $I_{i-1}$ into its left half $I^0_{i-1}$ and its right half $I^1_{i-1}$.
Since $k_i \in \S$, 
$\mu(I^0_{i-1})=
 \mu(I^1_{i-1})=2^{-i}$.

Choose $p_i \ge p_{i-1}$ such that $u_{p_i} <  2^{-2i}\varepsilon$.  

Set 

$
b_{k_i} = \begin{cases}
0 & \text{if } \measure{\Delta_{p_i} \cap I^0_{i-1}}+u_{p_i}  < 2^{-i} (\varepsilon + \sum_{j=1}^{i} 2^{\,j-1} u_{p_j}), \text{is verified first,}\\
1 & \text{if } \measure{\Delta_{p_i} \cap I^1_{i-1}}+u_{p_i}  < 2^{-i} (\varepsilon + \sum_{j=1}^{i} 2^{\,j-1} u_{p_j}), \text{is verified first.}
\end{cases}
$

Set  $I_i=I_{i-1}^{b_{k_i}}$.
\end{itemize}
\end{algorithm}
\bigskip

The following lemma proves the invariant of the algorithm: at every step $i$, the  interval $I_i$ contains a positive $\mu$-measure of numbers that are not in $\Delta$.

\begin{lemma}\label{lemma:positive}
For every $i\geq 0$, $\mu(I_i\setminus \Delta)>0$.
\end{lemma}
\begin{proof}
We prove the lemma by induction on the number of steps $i$.
Let 
\[T_i = \varepsilon + \sum_{j=1}^{i} 2^{\,j-1} u_{p_j}.
\]
\noindent Induction hypothesis (IH):
For every $i\ge 0$,
\[
\measure{\Delta_{p_i} \cap I_i} + u_{p_i} < \frac{1}{2^{i}}\,T_i.
\]

\noindent\textit{Base case $i=0$.}
Since $\mu(\Delta_{p_0}) \le s_{p_0}$, by subadditivity,
Hence
\[
\measure{\Delta_{p_0}\cap I_0} + u_{p_0} \le s_{p_0} + u_{p_0} = s < \varepsilon = T_0 .
\]
\textit{Inductive step.}
Assume (IH) for $i-1$.
Since $I^0_{i-1}\cup I^1_{i-1} = I_{i-1}$,
\[
\measure{\Delta_{p_i} \cap I^0_{i-1}} + \measure{\Delta_{p_i} \cap I^1_{i-1}} = \measure{\Delta_{p_i} \cap I_{i-1}} .
\]
Apply Lemma~\ref{lem:props}  point (\ref{item:4}) with $t = p_{i-1}$, $\ell = p_i$, $I = I_{i-1}$:
\[
\measure{\Delta_{p_i} \cap I_{i-1}} \le \measure{\Delta_{p_{i-1}} \cap I_{i-1}} + u_{p_{i-1}} - u_{p_i}.
\]
Adding $2u_{p_i}$,
\[
(\measure{\Delta_{p_i} \cap I^0_{i-1}} + u_{p_i}) + (\measure{\Delta_{p_i} \cap I^1_{i-1}} + u_{p_i})
\le \measure{\Delta_{p_{i-1}} \cap I_{i-1}} + u_{p_{i-1}} + u_{p_i}.
\]
By (IH) for $i-1$,
\[
\measure{\Delta_{p_{i-1}} \cap I_{i-1}} + u_{p_{i-1}} < \frac{1}{2^{\,i-1}} T_{i-1}.
\]
%
Hence
\[
(\measure{\Delta_{p_i} \cap I^0_{i-1}} + u_{p_i}) + 
(\measure{\Delta_{p_i} \cap I^1_{i-1}} + u_{p_i})
 < \frac{1}{2^{\,i-1}} T_{i-1} + u_{p_i}.
\]
Hence, if both, 
$\big(\measure{\Delta_{p_i} \cap I^0_{i-1}} +u_{p_i} \big)$
and 
$\big(\measure{\Delta_{p_i} \cap I^1_{i-1}} +u_{p_i} \big)
$
were greater than or equal to 
\[
\frac{1}{2} \Big(\frac{1}{2^{\,i-1}} T_{i-1} +u_{p_i}\Big) = \frac{1}{2^{i}} T_{i}
\]
we would have 
\[
(\measure{\Delta_{p_i} \cap I^0_{i-1}} + u_{p_i}) + 
(\measure{\Delta_{p_i} \cap I^1_{i-1}} + u_{p_i})
 \ge2\frac{1}{2^{i}} T_{i}=\frac{1}{2^{\,i-1}} T_{i-1} +u_{p_i}
\]
contradicting the inequality above.
Therefore, at least one of them is less than~$\frac{1}{2^{i}} T_i$.
The  interval  $I_i=I^{b_{k_i}}_{i-1}$ chosen by the algorithm  satisfies
\[
\measure{\Delta_{p_i} \cap I_i} + u_{p_i} < \frac{1}{2^{i}} T_i .
\]
Since $u_{p_i}< 2^{-2i}\varepsilon$,
\[
\sum_{j\geq 1} 2^{\,j-1} u_{p_j}
<\varepsilon \sum_{j\geq 1} 2^{-j-1} = \frac{\varepsilon}{2}.
\]
Thus, $T_i \le \varepsilon + \varepsilon/2 = (3/2)\varepsilon $ for all $i$.
Since we assumed $\varepsilon < 2/3$, from (IH) and Lemma~\ref{lem:props} point~(\ref{item:3}),
\[
\measure{\Delta \cap I_i} \le \measure{\Delta_{p_i} \cap I_i} + u_{p_i}
< \frac{\varepsilon ~3/2}{2^{i}} < 2^{-i} = \mu(I_i).
\]
We conclude that no $I_i$ is completely covered by $\Delta$.
\end{proof}

\begin{lemma}
Let positive rational $\varepsilon < 2/3$. The number $x=x(\varepsilon)$ defined by the algorithm  is normal to all odd bases.
\end{lemma}
\begin{proof}
The nested dyadic intervals $I_i$ shrink to a unique point $x\in\C$, that is, $x=\bigcap_{i\geq 1}I_i$.
By Lemma~\ref{lemma:positive}, for every $i\geq 0$, $\mu(I_i\setminus \Delta)>0$,
we conclude that $x\not\in\Delta$.
Take any odd $r\ge3$ and any integer $h\neq0$.
For all sufficiently large $m$, we have $|h|\le \log\log N_m$ and $x\notin\Lambda_{r,m}$, so
\[
\Bigl|\frac1{N_m}\sum_{j=1}^{N_m} e^{2\pi i r^{\,j} h x}\Bigr|
\le \frac{1}{\log\log N_m} \to 0 \quad\text{as } m\to\infty .
\]
Now let $M$ be an arbitrary positive integer.  Choose the unique $m$ such that
\[
N_{m-1} < M \le N_m ,
\]
where we set $N_0 = 1$.
Then we can write
\[
\frac{1}{M}\sum_{j=1}^{M} e^{2\pi i r^{\,j} h x}
= \frac{N_m}{M}\cdot\frac{1}{N_m}\sum_{j=1}^{N_m} e^{2\pi i r^{\,j} h x}
\;-\; \frac{1}{M}\sum_{j=M+1}^{N_m} e^{2\pi i r^{\,j} h x}.
\]
The absolute value of the second  summand is at most 
\[
(N_m-M)/{M}.
\]
Since 
$N_m =\lfloor 2^{m^{\alpha}}\rfloor$ is subexponential in $m$,  we have
 $N_m/N_{m-1} \to 1$.
 Since $M > N_{m-1}$, we have $N_m/M \le N_m/N_{m-1} \to 1$, hence $\frac{N_m-M}{M}~\to~0$.
 Thus, the second summand tends to~$0$.
Now we argue that the first summand also tends to~$0$. Given that  $x\not \in \Lambda_{r,m}$, we have 
$\left|\frac{1}{N_m}\sum_{j=1}^{N_m} e^{2\pi i r^{\,j} h x}\right|\leq 1/(\log\log N_m)$ which tends to $0$ and $m\to\infty$. Since $N_m/M \to 1$, we conclude
\[
\lim_{M\to\infty} \frac{1}{M}\sum_{j=1}^{M} e^{2\pi i r^{\,j} h x} = 0 .
\]
This is Weyl’s criterion for the sequence $(r^{\,j}x )_{j\ge 1}$ to be uniformly distributed modulo~$1$. As a result, $x$ is normal in base~$r$.
\end{proof}

\begin{lemma}
Assume  $\S$ is a computable sparse set  with a computable sparsity exponent $\rho$.
Assume  $\varepsilon $ is a  positive  real number such that $\varepsilon< 2/3$. Then, then the number $x=x(\varepsilon) $ defined by Algorithm ~\ref{alg:1} is computable.
\end{lemma}
\begin{proof}
The sparse set  $\S$ is computable, so there is a a computable enumeration of all its elements in increasing order. It also implies that the measure $\mu$ on dyadic intervals is computable.
By Remark~\ref{rem:1}, for each odd $r\geq 3$ and for each $m$, 
$\mu(\Lambda_{r,m})$ is computable.
Since the value $m_0(r)$ we have, for each $t$, $\mu(\Delta_t$) is is computable.
 By Remark~\ref{rem:3}, $u_t$ is computable. Given that  $u_t\to 0$,
 the algorithm can perform a finite search to find the first $p_i$ such that $u_{p_i}< 2^{-2i}\varepsilon$. 
  Since  $\Delta_{p_i}$ is a finite union of intervals with computable endpoints, so is $\Delta_{p_i}\cap I^0_{i-1}$.  The measure $\mu$ on this set is computable.  The choice of $b_{k_i}$ is the result of a comparison of computable reals.
  The decision whether a given  real number is strictly less than another is  computable unless they are equal.
By the proof of Lemma~\ref{lemma:positive}, at least   
one of  $I^0_{i-1}$ or  $I^1_{i-1}$ satisfies the strict inequality condition.
\end{proof}

\section*{Acknowledgements}
We thank Pablo Shmerkin for his explicit question to the first author  in Buenos Aires in 2014 on how to  construct numbers whose block frequencies converge to non-uniform limits in one base yet are normal in the multiplicatively independent bases. We also thank Benjamin Weiss for a valuable email exchange  on deterministic numbers in one base and normality in multiplicatively independent bases.

There are no grants supporting this work because since December 2023, governmental investment in science has been fully cut and the Argentine scientific system is undergoing a serious contraction.

\bibliographystyle{plain}
\bibliography{bibliography}

\end{document}